\documentclass[11pt]{amsart}

\usepackage[T1]{fontenc}
\usepackage[utf8]{inputenc}
\usepackage{amsmath,amssymb,amsthm,mathtools}
\usepackage{mathrsfs}
\usepackage{tikz}
\usepackage{enumitem}
\usepackage{needspace}
\usepackage{microtype}
\usepackage{relsize}
\usepackage[hidelinks]{hyperref}
\usepackage[backend=biber,style=alphabetic,maxbibnames=15,maxcitenames=6,dateabbrev=false,autolang=hyphen]{biblatex}

\renewcommand{\UrlFont}{\sffamily\smaller}
\renewbibmacro{in:}{\ifentrytype{article}{}{\printtext{\bibstring{in}\intitlepunct}}}
\DeclareFieldFormat{url}{{\UrlFont\url{#1}}}
\DeclareFieldFormat{urldate}{%
  (version \thefield{urlday}\addspace%
  \mkbibmonth{\thefield{urlmonth}}\addspace%
  \thefield{urlyear}\isdot)}
\DeclareFieldFormat{eprint:arxiv}{%
\ifhyperref
    {\href{http://arxiv.org/abs/#1}{%
    arXiv\addcolon\nolinkurl{#1}}\iffieldundef{eprintclass}{}{\printtext{ [\thefield{eprintclass}]}}}
    {arXiv\addcolon\nolinkurl{#1}\iffieldundef{eprintclass}{}{\printtext{ [\thefield{eprintclass}]}}}}
\DeclareFieldFormat{eprint:arXiv}{%
\ifhyperref
    {\href{http://arxiv.org/abs/#1}{%
    arXiv\addcolon\nolinkurl{#1}}\iffieldundef{eprintclass}{}{\printtext{ [\thefield{eprintclass}]}}}
    {arXiv\addcolon\nolinkurl{#1}\iffieldundef{eprintclass}{}{\printtext{ [\thefield{eprintclass}]}}}}
\defbibheading{mgtbibliography}{\begin{center}\scshape References\end{center}}
\AtBeginBibliography{\small\setlength{\emergencystretch}{3em}}

\title{Modal group theory: homomorphisms}
\author{Wojciech Aleksander Wo\l oszyn}
\address[Wojciech Aleksander Wo\l oszyn]
{Mathematical Institute, University of Oxford, Andrew Wiles Building, Radcliffe Observatory Quarter, Woodstock Road, Oxford, OX2 6GG, United Kingdom \&\ St Hilda's College, Cowley Place, Oxford, OX4 1DY, United Kingdom}
\email{wojciech@woloszyn.org}
\urladdr{https://woloszyn.org}

\newtheorem{theorem}{Theorem}[section]
\newtheorem{lemma}[theorem]{Lemma}
\newtheorem{proposition}[theorem]{Proposition}
\newtheorem{corollary}[theorem]{Corollary}
\theoremstyle{definition}
\newtheorem{definition}[theorem]{Definition}

\theoremstyle{remark}
\newtheorem{remark}[theorem]{Remark}

\newsavebox{\mgthompossiblebox}
\newsavebox{\mgthomnecessarybox}
\sbox{\mgthompossiblebox}{\tikz[scale=.6ex/1cm,baseline=-.6ex,rotate=45,line width=.1ex]{\draw (-1,-1) rectangle (1,1);}}
\sbox{\mgthomnecessarybox}{\tikz[scale=.6ex/1cm,baseline=-.6ex,line width=.1ex]{\draw (-1,-1) rectangle (1,1);}}
\DeclareMathOperator{\possible}{\text{\usebox{\mgthompossiblebox}}}
\DeclareMathOperator{\necessary}{\text{\usebox{\mgthomnecessarybox}}}
\newcommand{\satisfies}{\models}
\newcommand{\theoryf}[1]{{\rm #1}}

\newcommand{\GrpHom}{\mathrm{Grp}_{\to}}
\newcommand{\Lgrp}{\mathcal L_{\mathrm{Grp}}}
\newcommand{\LgrpModal}{\Lgrp^{\possible}}
\newcommand{\LgrpM}{\LgrpModal}
\newcommand{\Val}{\mathrm{Val}}
\newcommand{\Z}{\mathbb{Z}}
\newcommand{\N}{\mathbb{N}}
\newcommand{\gen}[1]{\langle #1\rangle}
\newcommand{\fpmodels}{\models_{\mathrm{fp,hom}}}
\newcommand{\hommodels}{\models_{\mathrm{hom}}}
\newcommand{\ctblhommodels}{\models_{\mathrm{ctbl,hom}}}
\newcommand{\Pow}[1]{\mathcal P(#1)}

\newcommand{\SFourTwo}{\ensuremath{\theoryf{S4.2}}}
\newcommand{\SFive}{\ensuremath{\theoryf{S5}}}
\newcommand{\Godel}{G\"odel}
\DeclareMathOperator{\ord}{ord}

\DeclareMathOperator{\ApplyPred}{Apply}
\DeclareMathOperator{\HomPred}{Hom}
\DeclareMathOperator{\GroupPred}{Group}
\DeclareMathOperator{\ValPred}{Val}
\DeclareMathOperator{\EltPred}{Elt}
\DeclareMathOperator{\FPGroupPred}{FPGroup}
\DeclareMathOperator{\TuplePred}{Tuple}

\begin{document}

\begin{abstract}
I investigate modal group theory for arbitrary homomorphisms. Possibility is interpreted by the existence of a group homomorphism out of the given group, so the semantics is governed by the possibility of collapse: elements may be identified, parameters may be killed, and new relations may hold in the target. I show that the modal language nevertheless expresses cyclic subgroup membership, subgroup generation by a fixed finite tuple, cyclicity, finite generation by a fixed number of elements, and torsion. I use these definability results to interpret arithmetic, and prove that, as sets of G\"odel numbers, the homomorphic modal theory of finitely presented groups is computably isomorphic to true arithmetic. I also analyze propositional modal validities: sentential validities are exactly \SFive, the trivial group has exact parameter-validities \SFive, and uniformly prime-indivisible groups have exact parameter-validities \SFourTwo.
\end{abstract}

\maketitle

\section{Introduction}

Modal group theory investigates categories of groups from a modal perspective: one studies what becomes possible or necessary when a group is allowed to move along a chosen class of group-theoretic maps. In this paper the maps are arbitrary homomorphisms. Thus, for a group $G$ and a tuple $\bar g$ from $G$, I write $G\hommodels\possible\varphi[\bar g]$ if there is a group homomorphism $h\colon G\to H$ such that $H\satisfies\varphi[h(\bar g)]$. Necessity is the corresponding universal quantifier over all homomorphisms out of $G$. I work in the inverse-free group language $\Lgrp=\{\cdot,e\}$, and I write $\LgrpModal$ for its modal expansion by the modal operators $\possible$ and $\necessary$.

The homomorphism semantics is governed by the possibility of collapse. A homomorphism may identify elements, kill parameters, and impose new relations in the target. This already produces simple modal phenomena. Since every group admits a homomorphism to every finite cyclic group, the statements saying that the target group has exactly $n$ elements behave like finite dials: exactly one value holds at a time, and from any group one can move to a group making any chosen value hold. With parameters, the assertion $a=e$ is a button: it can be made permanently true by a homomorphism killing $a$, and once true it remains true along all further homomorphisms. These examples are elementary, but they already explain why the homomorphism modal logic is shaped by collapse rather than by extension.

This is the basic contrast with the embedding semantics of the companion paper \emph{Modal group theory}~\cite{WoloszynEmbedding}. There, possibility means passing to overgroups, so a group can only see upwards. Here, possibility means mapping to arbitrary homomorphic targets and then continuing from there. The modal language therefore sees not only ways of enlarging a group-theoretic situation, but also ways of quotienting, collapsing, and reinterpreting it through homomorphisms.

Despite this collapse, a substantial part of the expressive power discovered in the embedding setting survives. The modal language for homomorphisms expresses cyclic subgroup membership, subgroup generation by a fixed finite tuple, cyclicity, finite generation by a fixed number of elements, and torsion. These definability results are strong enough to interpret arithmetic. In particular, the homomorphic modal theory of finitely presented groups is computably isomorphic to true arithmetic, while the homomorphic modal theory of countable groups under homomorphisms between countable groups is one--one reducible to true second-order arithmetic.

\Needspace{16\baselineskip}
\begin{samepage}
The main results are:
\begin{enumerate}
    \item Many group-theoretic properties inexpressible in the first-order language of groups become expressible in modal group theory for homomorphisms. For example, membership in cyclic subgroups and in subgroups generated by a fixed finite tuple, cyclicity, being generated by at most a fixed number of elements, and torsion are expressible.
    \item As sets of G\"odel numbers, the theory of true arithmetic is computably isomorphic to the homomorphic modal theory of finitely presented groups.
    \item The sentential propositional modal validities of groups under homomorphisms constitute precisely \SFive{}, while the formulaic parameter-validities of uniformly prime-indivisible groups constitute precisely \SFourTwo.
\end{enumerate}
\end{samepage}
I also analyze the trivial group, which validates \SFive{} even with parameters.

\medskip

Modal group theory is part of the larger \emph{modal model theory} program, which investigates modal principles and validities in categories of first-order structures. Hamkins and I introduced modal model theory, with modal graph theory as the first case study~\cite{HW24}. In my work on the modal theory of the category of sets, I develop the broader concrete-categorical perspective and extend this semantics beyond the original potentialist systems~\cite{WSet}.

\section*{Acknowledgments}
I thank Joel David Hamkins, Ehud Hrushovski, and Goh Jun Le for many helpful discussions related to modal group theory. I am grateful to Martin Bays for suggesting the amalgamated product $G*_{\gen{a}}A_p$ used in Lemma~\ref{lem:one-prime-step}, which is the key construction in the proof of Theorem~\ref{thm:hom-exact-s42}.

\section{Expressive power of modal group theory}\label{section:expressive}

We now isolate the basic group-theoretic notions that remain modally definable when accessibility is given by arbitrary homomorphisms. The formal group language is the inverse-free language $\Lgrp=\{\cdot,e\}$, and $\LgrpModal$ denotes its modal expansion. Whenever inverse notation appears below, it is only metanotation: thus $u=x^{-1}$ is shorthand for $xu=e\wedge ux=e$.

The embedding paper begins with the modal definition of equality of order. The particular conjugacy-based formula used there,
\[
\possible\exists s\,(s^{-1}xs=y),
\]
does not transfer to homomorphisms: every group maps to the trivial group, where the images of $x$ and $y$ are conjugate. I therefore begin with cyclic subgroup membership, which is the first embedding argument that survives unchanged.

\subsection{Cyclic subgroup membership and finite generation}

We begin with the standard HNN-extension lemma used to detect cyclic subgroup membership; for background on HNN extensions and Britton's lemma, see \cite{HNN49,LyndonSchupp}.

\begin{lemma}\label{lem:hnn-centralizer}
Let $A\leq G$, and form the HNN extension
\[
G^*=\gen{G,s\mid s\gamma=\gamma s\text{ for all }\gamma\in A}.
\]
Then
\[
C_{G^*}(s)\cap G=A.
\]
\end{lemma}

\begin{proof}
Every element of $A$ commutes with $s$ by construction. Conversely, let $g\in G\setminus A$. Then the word $s^{-1}gsg^{-1}$ is reduced, so Britton's lemma implies that it is nontrivial in $G^*$. Thus $s^{-1}gs\neq g$, equivalently $sg\neq gs$. Hence no element of $G\setminus A$ lies in $C_{G^*}(s)$.
\end{proof}

\begin{theorem}\label{thm:hom-cyclic-membership}
The relation of belonging to the cyclic subgroup generated by another element is expressible in modal group theory for homomorphisms. More precisely, the formula
\[
y\in\gen{x}\;:=\;\necessary\forall t\,(tx=xt\to ty=yt)
\]
satisfies
\[
G\hommodels y\in\gen{x}
\quad\text{iff}\quad
 y\in\gen{x}\text{ in }G.
\]
\end{theorem}

\begin{proof}
If $y\in\gen{x}$, say $y=x^k$ for some $k\in\Z$, then for every homomorphism $h\colon G\to H$ any element of $H$ commuting with $h(x)$ also commutes with $h(x)^k=h(y)$. Hence the displayed sentence holds in every homomorphic image of $G$.

Conversely, suppose
\[
G\hommodels \necessary\forall t\,(tx=xt\to ty=yt).
\]
Form the HNN extension
\[
G^*=\gen{G,s\mid s\gamma=\gamma s\text{ for all }\gamma\in\gen{x}}.
\]
The canonical embedding $\iota\colon G\hookrightarrow G^*$ is, in particular, a homomorphism, so the displayed sentence holds in $G^*$ with $\iota(x)$ and $\iota(y)$ in place of $x$ and $y$. Since $s\iota(x)=\iota(x)s$ in $G^*$, instantiating $t=s$ yields $s\iota(y)=\iota(y)s$. By Lemma~\ref{lem:hnn-centralizer}, the elements of $\iota(G)$ commuting with $s$ are exactly the elements of $\iota(\gen{x})$. Hence $\iota(y)\in\iota(\gen{x})$, and therefore $y\in\gen{x}$ in $G$.
\end{proof}

\begin{corollary}\label{cor:hom-cyclic}
Cyclicity is expressible in modal group theory for homomorphisms.
\end{corollary}

\begin{proof}
A group is cyclic exactly when
\[
\exists x\,\forall y\,(y\in\gen{x}),
\]
and the relation on the right is expressed by Theorem~\ref{thm:hom-cyclic-membership}.
\end{proof}

\begin{theorem}\label{thm:hom-tuple-generation}
For each $n\geq 1$, the relation that an element lies in the subgroup generated by $x_1,\dots,x_n$ is expressible in modal group theory for homomorphisms. More precisely,
\[
y\in\gen{x_1,\dots,x_n}
\;:=\;
\necessary\forall t\!\left(\bigwedge_{i=1}^n tx_i=x_it\to ty=yt\right)
\]
satisfies
\[
G\hommodels y\in\gen{x_1,\dots,x_n}
\quad\text{iff}\quad
 y\in\gen{x_1,\dots,x_n}\text{ in }G.
\]
\end{theorem}

\begin{proof}
If $y\in\gen{x_1,\dots,x_n}$, then $y$ is a word in the letters $x_i^{\pm1}$. Therefore for every homomorphism $h\colon G\to H$, any element of $H$ commuting with all $h(x_i)$ also commutes with $h(y)$, so the displayed sentence holds in every homomorphic image of $G$.

Conversely, suppose the displayed sentence holds homomorphically in $G$, and form the HNN extension
\[
G^*=\gen{G,s\mid s\gamma=\gamma s\text{ for all }\gamma\in\gen{x_1,\dots,x_n}}.
\]
The canonical embedding $\iota\colon G\hookrightarrow G^*$ is a homomorphism, so the displayed sentence holds in $G^*$ with $\iota(x_1),\dots,\iota(x_n),\iota(y)$. Since $s$ commutes with every $\iota(x_i)$, instantiating $t=s$ yields $s\iota(y)=\iota(y)s$. Lemma~\ref{lem:hnn-centralizer} now implies that $\iota(y)\in \iota(\gen{x_1,\dots,x_n})$, and hence $y\in\gen{x_1,\dots,x_n}$ in $G$.
\end{proof}

\begin{corollary}\label{cor:hom-n-generated}
For every $n\geq 1$, the property of being generated by at most $n$ elements is expressible in modal group theory for homomorphisms.
\end{corollary}

\begin{proof}
A group is generated by at most $n$ elements exactly when
\[
\exists x_1\cdots\exists x_n\,\forall y\,(y\in\gen{x_1,\dots,x_n}),
\]
and the subgroup-membership relation is expressed by Theorem~\ref{thm:hom-tuple-generation}.
\end{proof}

\subsection{Torsion}

\begin{theorem}\label{thm:hom-torsion-element}
The property of being a torsion element is expressible in the modal group theory for homomorphisms. There is a formula $\ord(x)<\infty$ such that for every group $G$,
\[
G\hommodels \ord(x)<\infty
\quad\text{iff}\quad
x\text{ has finite order in }G.
\]
\end{theorem}

\begin{proof}
Let $\tau_0(x)$ be any first-order formula expressing that $x$ has order $1$, $2$, $3$, $4$, or $6$. Define
\[
\ord(x)<\infty\;:= \tau_0(x)\vee \exists y\,(y\neq x\wedge xy\neq e\wedge x\in\gen{y}\wedge y\in\gen{x}).
\]
We show that this formula has the desired meaning.

Suppose first that $x$ has finite order $n$. If $n\in\{1,2,3,4,6\}$, then $\tau_0(x)$ holds. Otherwise Euler's totient function satisfies $\varphi(n)\geq 3$, so there is some integer $k$ relatively prime to $n$ with $k\not\equiv \pm1\pmod n$. Let $y=x^k$. Then $y\neq x$ and $xy\neq e$. Since $\gcd(k,n)=1$, B\'ezout gives integers $a,b$ with $ak+bn=1$, and hence
\[
x=x^{ak+bn}=x^{ak}=y^a.
\]
So $x\in\gen{y}$ and $y\in\gen{x}$.

Conversely, suppose there exists some $y$ with $y\neq x$, $xy\neq e$, $x\in\gen{y}$, and $y\in\gen{x}$. Then $\gen{x}=\gen{y}$. If $x$ had infinite order, then $\gen{x}\cong \Z$, and the only generators of $\Z$ are $1$ and $-1$. Translating back, the only generators of $\gen{x}$ would be $x$ and $x^{-1}$, contrary to $y\neq x$ and $xy\neq e$. Hence $x$ has finite order.
\end{proof}

\begin{corollary}\label{cor:hom-torsion-group}
The set of torsion elements is definable in modal group theory for homomorphisms, and the property of being a torsion group is expressible by the sentence
\[
\forall x\,(\ord(x)<\infty).
\]
\end{corollary}

\begin{proof}
Theorem~\ref{thm:hom-torsion-element} defines the torsion elements. A group is torsion exactly when every element is torsion.
\end{proof}

\begin{remark}\label{rem:same-order-fails}
This is not a nondefinability theorem for equality of order in the homomorphism semantics. It says only that the embedding proof and its conjugacy formula give no useful homomorphism definition.
\end{remark}

\section{Non-expressivity in first-order group theory}\label{section:nonexpressivity}

The properties isolated in Section~\ref{section:expressive} are not definable in ordinary first-order group theory. The proofs use standard compactness, upward L\"owenheim--Skolem, and ultraproduct arguments from classical model theory; see, for example, \cite[Chapters~4 and~6]{ChangKeisler}.

\begin{theorem}\label{thm:hom-not-fo-expressible}
None of the following properties is expressible in first-order group theory:
\begin{enumerate}[label=(\arabic*)]
    \item for each fixed $n\geq 1$, element $y$ belongs to the subgroup generated by $x_1,\dots,x_n$;
    \item the group is cyclic;
    \item for each fixed $n\geq 1$, the group is generated by at most $n$ elements;
    \item element $x$ is torsion;
    \item the group is torsion.
\end{enumerate}
Moreover, the set of torsion elements is not first-order definable in the pure language of groups.
\end{theorem}

\begin{proof}
For (1), fix $n\geq 1$ and suppose there were a first-order formula $\varphi(y,x_1,\dots,x_n)$ such that
\[
G\models \varphi[y,x_1,\dots,x_n]
\quad\text{iff}\quad
 y\in\gen{x_1,\dots,x_n}\text{ in }G.
\]
Consider the theory in the language of groups with constants $c,d_1,\dots,d_n$ consisting of the group axioms, the sentence $\varphi(c,d_1,\dots,d_n)$, and the inequalities
\[
c\neq w(d_1,\dots,d_n)
\]
for every group word $w$ in $n$ variables. This theory is inconsistent. But every finite fragment is satisfiable in the free group $F_n=\gen{a_1,\dots,a_n}$: interpret $d_i$ by $a_i$ and choose $c$ to be a word in the $a_i$ different from the finitely many forbidden words. Then $c\in\gen{d_1,\dots,d_n}$, while all finitely many inequalities hold. Compactness gives a contradiction.

For (2), if cyclicity were first-order expressible, then the infinite cyclic group $\Z$ would satisfy a first-order sentence whose models were exactly the cyclic groups. By the upward L\"owenheim--Skolem theorem, that sentence would have models of arbitrarily large infinite cardinalities. But there are no uncountable cyclic groups.

For (3), fix $n\geq 1$. If being generated by at most $n$ elements were first-order expressible, then the free group $F_n$ would satisfy a first-order sentence whose models were exactly the groups generated by at most $n$ elements. By the upward L\"owenheim--Skolem theorem, that sentence would have an uncountable model. But every group generated by at most $n$ elements is countable, since its elements are represented by finite words over a finite alphabet.

For (4), suppose there were a first-order formula $\tau(x)$ such that
\[
G\models \tau[x] \quad\text{iff}\quad x\text{ is torsion in }G.
\]
In the language of groups with a constant $c$, consider the theory consisting of the group axioms, the sentence $\tau(c)$, and the inequalities $c^m\neq 1$ for all positive integers $m$. This theory is not satisfiable, but every finite fragment is satisfiable in a sufficiently large finite cyclic group, interpreting $c$ as a generator. Hence compactness fails.

For (5), let $G_m$ be the finite cyclic group of order $m!$. Each $G_m$ is a torsion group. But in the ultraproduct $\prod_m G_m/U$ for a nonprincipal ultrafilter $U$, the class of generators has infinite order, since the orders of those generators are unbounded. Hence the ultraproduct is not a torsion group. By \L o\'s theorem, torsion-grouphood is not first-order expressible.

The final assertion follows immediately from (4).
\end{proof}

\begin{corollary}\label{cor:hom-compactness-ls}
Modal group theory for homomorphisms fails both compactness and the upward L\"owenheim--Skolem property.
\end{corollary}

\begin{proof}
Compactness fails because Theorem~\ref{thm:hom-torsion-element} expresses ``$c$ is torsion'' by a single modal formula, while the theory used in the proof of Theorem~\ref{thm:hom-not-fo-expressible}(4) is finitely satisfiable but not satisfiable.

The upward L\"owenheim--Skolem property fails because cyclicity is modally expressible by Corollary~\ref{cor:hom-cyclic}, the infinite cyclic group is a model of that modal sentence, and there are no uncountable cyclic groups.
\end{proof}

\section{The complexity of modal group theory}\label{section:complexity}

We now use the expressive power established above to interpret arithmetic in modal group theory for homomorphisms. The organization follows the embedding companion paper: first I interpret Presburger arithmetic with divisibility using an infinite-order parameter, and then I use purely arithmetic definability facts to recover multiplication.

\subsection{Arithmetic with an infinite-order parameter}

\begin{theorem}\label{thm:presburger-hom}
Presburger arithmetic with divisibility over the integers is interpretable in the modal group theory for homomorphisms with a parameter of infinite order. Let $G$ be a group with an element $g$ of infinite order. Then there is a computable translation $\phi\mapsto \phi^*(x)$ from arithmetical sentences to formulas of modal group theory such that
\[
\langle \Z,+,\mid,0,1\rangle\models \phi
\quad\text{iff}\quad
G\hommodels \phi^*[g].
\]
\end{theorem}

\begin{proof}
Represent the integer $i\in\Z$ by the group element $g^i$. Then $0$ is represented by the identity and $1$ by $g$. Addition is interpreted by group multiplication on $\gen{g}$:
\[
i+j=k \quad\text{iff}\quad g^ig^j=g^k.
\]
Divisibility is interpreted exactly as before:
\[
i\mid j \quad\text{iff}\quad g^j\in\gen{g^i}.
\]
The modal ingredient needed here is the predicate $y\in\gen{x}$, which is expressible by Theorem~\ref{thm:hom-cyclic-membership}.

We translate arithmetic formulas after first rewriting them so that their atomic formulas are of the forms $x+y=z$, $x=y$, $x\mid y$, $x=0$, or $x=1$. The constant $1$ is represented by the parameter $g$. Quantifiers are guarded to the copy of $\Z$ inside $\gen{g}$. Thus the recursive clauses are
\begin{align*}
(\neg \phi)^*&=\neg\phi^*,\\
(\phi\wedge\psi)^*&=\phi^*\wedge\psi^*,\\
(\exists y\,\phi)^*&=\exists y\,(y\in\gen{g}\wedge\phi^*),
\end{align*}
with the obvious atomic translations. The induction hypothesis is applied only to assignments of the free arithmetic variables into the interpreted copy $\gen{g}$. The guarded quantifier clause above is precisely what preserves that invariant. The correctness verification is otherwise identical to the embedding case, with Theorem~\ref{thm:hom-cyclic-membership} replacing its embedding analogue.
\end{proof}

\begin{lemma}\label{lem:squaring}
The squaring function is definable in Presburger arithmetic with divisibility over the integers,
\[
\langle \Z,+,\mid,0,1\rangle.
\]
That is, there is an arithmetic formula $\sigma(x,y)$ such that
\[
\langle \Z,+,\mid,0,1\rangle\models \sigma[x,y]
\quad\text{iff}\quad
y=x^2.
\]
\end{lemma}

\begin{proof}
Over the integers, two integers have the same set of multiples exactly when they have the same absolute value. For all integers $x$, the number $x$ is relatively prime both to $x+1$ and to $x-1$. Hence the common multiples of $x$ and $x+1$ are exactly the multiples of $x(x+1)$, and the common multiples of $x$ and $x-1$ are exactly the multiples of $x(x-1)$.

To keep the formula literally in the language $\{+,\mid,0,1\}$, introduce variables $u$ and $t$ to stand for $x-1$ and $y-x$. Define $\sigma(x,y)$ to be
\[
\begin{aligned}
\exists u\exists t\Bigl(&u+1=x\wedge x+t=y\\
&\wedge \forall z\bigl(((x\mid z)\wedge((x+1)\mid z))\leftrightarrow((y+x)\mid z)\bigr)\\
&\wedge \forall z\bigl(((x\mid z)\wedge(u\mid z))\leftrightarrow(t\mid z)\bigr)\Bigr).
\end{aligned}
\]
This says that $y+x$ has the same multiples as $x(x+1)$ and that $y-x$ has the same multiples as $x(x-1)$.

Thus $\sigma(x,y)$ holds iff there are signs $\varepsilon,\delta\in\{\pm1\}$ such that
\[
y+x=\varepsilon x(x+1),\qquad y-x=\delta x(x-1).
\]
If $(\varepsilon,\delta)=(1,1)$, then immediately $y=x^2$. If $(\varepsilon,\delta)=(1,-1)$, then adding the equations gives $y=x$, and the first equation forces $x=0$ or $x=1$; in both cases $y=x^2$. If $(\varepsilon,\delta)=(-1,1)$, then $y=-x$, and the first equation forces $x=0$ or $x=-1$; again $y=x^2$. Finally, if $(\varepsilon,\delta)=(-1,-1)$, then $y=-x^2$, and the first equation then forces $x=0$, so $y=x^2$. Hence $\sigma$ defines exactly the squaring relation.
\end{proof}

\begin{lemma}\label{lem:multiplication}
Multiplication is definable in $\langle \Z,+,\mid,0,1\rangle$. That is, there is an arithmetic formula $\mu(x,y,z)$ such that
\[
\langle \Z,+,\mid,0,1\rangle\models \mu[x,y,z]
\quad\text{iff}\quad
z=x\cdot y.
\]
\end{lemma}

\begin{proof}
By Lemma~\ref{lem:squaring}, squaring is definable. Using
\[
(x+y)^2=x^2+xy+yx+y^2,
\]
we see that $z=xy$ iff
\[
2z+x^2+y^2=(x+y)^2.
\]
Thus multiplication is defined by the formula
\[
\exists u\exists v\exists w\bigl(\sigma(x,u)\wedge\sigma(y,v)\wedge\sigma(x+y,w)\wedge z+z+u+v=w\bigr),
\]
where $\sigma$ is the squaring formula from Lemma~\ref{lem:squaring}. This formula uses only $+$, $\mid$, $0$, and $1$.
\end{proof}

\begin{corollary}\label{cor:ring-with-parameter-hom}
The ring of integers is interpretable in modal group theory for homomorphisms with a parameter of infinite order.
\end{corollary}

\begin{proof}
Combine Theorem~\ref{thm:presburger-hom} with Lemma~\ref{lem:multiplication}, which defines multiplication in $\langle\Z,+,\mid,0,1\rangle$.
\end{proof}

\begin{theorem}\label{thm:ring-of-integers-hom}
The theory of the ring of integers is interpretable in modal group theory for homomorphisms. There is a computable translation $\phi\mapsto \phi^{\dagger}$ from sentences of the ring language $\{+,\cdot,0,1\}$ to modal group-theoretic assertions such that
\[
\begin{aligned}
\langle \Z,+,\cdot,0,1\rangle\models \phi
\quad\text{iff}\quad
\phi^{\dagger}\text{ holds}\\
\text{at all groups in the homomorphism semantics.}
\end{aligned}
\]
Equivalently, $\phi^{\dagger}$ holds at all groups if and only if it holds at the trivial group.
\end{theorem}

\begin{proof}
By Theorem~\ref{thm:hom-torsion-element}, the condition that a parameter have infinite order is expressible as $\neg(\ord(x)<\infty)$. Moreover, Theorem~\ref{thm:hom-cyclic-membership} expresses that another element belongs to the cyclic subgroup generated by that parameter. Thus we can dispose of the parameter by quantifying over it inside the scope of $\possible$.

Let $\phi\mapsto \phi^*(x)$ be the parameter-dependent translation from Corollary~\ref{cor:ring-with-parameter-hom}. Define
\[
\phi^{\dagger}:=\possible\exists x\,\bigl(\neg(\ord(x)<\infty)\wedge \phi^*(x)\bigr).
\]
Then
\[
\langle \Z,+,\cdot,0,1\rangle\models \phi
\quad\text{iff}\quad
G\hommodels \phi^{\dagger}
\]
for all groups $G$. The verification is exactly the same as in the parameter-dependent interpretation, using the two definability results just cited.

Finally, in the homomorphism semantics the truth of a sentence of the form $\possible\psi$ does not depend on the base group: if some group $H$ satisfies $\psi$, then for every group $G$ there is a homomorphism $G\to H$ (for instance the trivial homomorphism), so $G\hommodels \possible\psi$. In particular, $\phi^{\dagger}$ holds at all groups if and only if it holds at the trivial group.
\end{proof}

Recall that a set $A$ is \emph{one--one reducible} to a set $B$ if there is a total injective computable function $f\colon\N\to\N$ such that $n\in A$ exactly when $f(n)\in B$. A \emph{computable isomorphism} between $A$ and $B$ is a computable bijection $p\colon\N\to\N$ such that $n\in A$ exactly when $p(n)\in B$.

All theories below are regarded as subsets of $\N$ via a fixed G\"odel coding of sentences; codes that are not sentences of the relevant language are not elements of the theory. We shall use the following totalization convention for sentence translations. Suppose that $\tau$ is a computable translation from source-language sentences to target-language sentences, injective on sentence codes, and truth-preserving for the source and target theories under consideration. Before using $\tau$ as a one--one reduction, replace it harmlessly by the truth-equivalent tagged translation
\[
\tau^{\sharp}(\sigma):=(\exists z\,z=z)\wedge \tau(\sigma),
\]
with the displayed parenthesization fixed as part of the coding. Let $\chi_n$ be the right-associated conjunction of $n+1$ copies of the true target-language sentence $\exists z\,z=z$, and set
\[
\delta_n:=(\forall z\,z\neq z)\wedge \chi_n.
\]
The sentences $\delta_n$ are pairwise syntactically distinct target-language contradictions, and their displayed outer form is disjoint from the image of $\tau^{\sharp}$. Define a total function on all natural numbers by
\[
F_{\tau}(n)=
\begin{cases}
\ulcorner\tau^{\sharp}(\sigma)\urcorner, & \text{if }n=\ulcorner\sigma\urcorner\text{ for a source-language sentence }\sigma,\\
\ulcorner\delta_n\urcorner, & \text{if }n\text{ is not a source-language sentence code.}
\end{cases}
\]
The predicate ``$n$ codes a source-language sentence'' is decidable, so $F_{\tau}$ is total computable. It is injective by the injectivity of $\tau$ on sentence codes, the pairwise distinctness of the $\delta_n$, and the disjoint outer tags. Moreover, non-sentence codes are outside the source theory, while each $\delta_n$ is outside the target theory. Thus every such sentence translation induces a genuine one--one reduction of sets of natural numbers. We use Myhill's isomorphism theorem in the form that if $A\leq_1 B$ and $B\leq_1 A$, then there is a computable isomorphism between $A$ and $B$ \cite{Myhill55}.

By \emph{true arithmetic} we mean the complete first-order theory of $\langle\N,+,\cdot,0,1\rangle$.

\begin{lemma}\label{lem:true-arithmetic-to-Z-hom}
There is a computable injective translation $\sigma\mapsto\sigma^{\Z}$ from sentences of first-order arithmetic to sentences of the ring language $\{+,\cdot,0,1\}$ such that
\[
\langle\N,+,\cdot,0,1\rangle\models\sigma
\quad\text{if and only if}\quad
\langle\Z,+,\cdot,0,1\rangle\models\sigma^{\Z}.
\]
\end{lemma}

\begin{proof}
Inside $\langle\Z,+,\cdot,0,1\rangle$ define
\[
\mathrm{Nat}(x):=\exists a\exists b\exists c\exists d\,\bigl(x=a^2+b^2+c^2+d^2\bigr).
\]
By Lagrange's four-square theorem, $\mathrm{Nat}(x)$ holds exactly of the nonnegative integers. Let $\sigma\mapsto\sigma^{\circ}$ be the usual relativized translation: atomic formulas of first-order arithmetic are translated verbatim into the ring language, Boolean connectives are translated recursively, and quantifiers are relativized to $\mathrm{Nat}(x)$:
\[
(\exists x\,\psi)^{\circ}:=\exists x\,(\mathrm{Nat}(x)\wedge \psi^{\circ}),\qquad
(\forall x\,\psi)^{\circ}:=\forall x\,(\mathrm{Nat}(x)\to\psi^{\circ}).
\]
Since addition and multiplication on the definable domain $\mathrm{Nat}(\Z)$ agree with addition and multiplication on $\N$, we have
\[
\langle\N,+,\cdot,0,1\rangle\models\sigma
\quad\text{if and only if}\quad
\langle\Z,+,\cdot,0,1\rangle\models\sigma^{\circ}.
\]
To make the sentence map injective, fix a computable family $(\alpha_n)_{n\in\N}$ of pairwise syntactically distinct true ring sentences; for definiteness, let $\alpha_n$ be the sentence $\overline n=\overline n$. Define
\[
\sigma^{\Z}:=\alpha_{\ulcorner\sigma\urcorner}\wedge \sigma^{\circ}.
\]
The added tag is true in $\langle\Z,+,\cdot,0,1\rangle$, so it preserves the displayed equivalence. The map $\sigma\mapsto\sigma^{\Z}$ is computable and injective on G\"odel codes because the left conjunct syntactically records the code of $\sigma$.
\end{proof}

\begin{corollary}\label{cor:true-arithmetic-hom}
The theory of true arithmetic is one--one reducible to modal group theory for homomorphisms.
\end{corollary}

\begin{proof}
Let $\sigma$ be a first-order arithmetic sentence and let $n=\ulcorner\sigma\urcorner$. Write
\[
\top_n^{\mathrm{grp}}:=\underbrace{(e=e)\wedge\cdots\wedge(e=e)}_{n+1\text{ conjuncts}}
\]
and define
\[
\widetilde\sigma:=\top_n^{\mathrm{grp}}\wedge (\sigma^{\Z})^{\dagger},
\]
where $\sigma^{\Z}$ is the ring sentence from Lemma~\ref{lem:true-arithmetic-to-Z-hom} and $(\sigma^{\Z})^{\dagger}$ is the translation from Theorem~\ref{thm:ring-of-integers-hom}. The map $\sigma\mapsto\widetilde\sigma$ is computable and injective on G\"odel codes, because the padding $\top_n^{\mathrm{grp}}$ remembers $n$ syntactically. Moreover, $\top_n^{\mathrm{grp}}$ is valid at every group, so
\[
\begin{aligned}
\langle\N,+,\cdot,0,1\rangle\models\sigma
&\quad\text{iff}\quad \langle\Z,+,\cdot,0,1\rangle\models\sigma^{\Z}\\
&\quad\text{iff}\quad \widetilde\sigma\text{ is homomorphically valid in all groups.}
\end{aligned}
\]
By the totalization convention above, this injective sentence translation induces a one--one reduction of true arithmetic to modal group theory for homomorphisms.
\end{proof}

\subsection{Finitely presented groups}

We write $\fpmodels$ for modal satisfaction in the category whose objects are finitely presented groups and whose arrows are arbitrary homomorphisms. Throughout, a word on $n$ generators is coded by the G\"odel number of its letter-sequence. If $e$ codes a finite presentation, $\EltPred(e,w)$ is the decidable predicate asserting that $w$ is a word on the generators named in the presentation coded by $e$.

The first theorem in this subsection arithmetizes modal truth over finitely presented groups. I then establish the cyclic-subgroup and torsion definability results needed for the converse reduction from true arithmetic in the category of finitely presented groups and homomorphisms.

\begin{theorem}\label{thm:fp-hom}
For every code $e$ of a finitely presented group $G_e$, every tuple of word-codes $\bar w=\langle w_1,\dots,w_k\rangle$ of the same length as $\bar x$ such that $\bigwedge_{i=1}^k \EltPred(e,w_i)$, and every formula $\phi(\bar x)$ in the modal language of groups (with the modal operators interpreted via homomorphisms between finitely presented groups), there is a computable arithmetical formula $\phi^{\dagger}(e,\bar w)$ such that
\[
(G_e,\bar w)\fpmodels \phi
\quad\text{iff}\quad
\N\models \phi^{\dagger}(e,\bar w).
\]
Consequently, the homomorphic modal theory of finitely presented groups is uniformly interpretable in true arithmetic.
\end{theorem}

Here each word-code $w_i$ names the element of $G_e$ represented by that word. Thus $(G_e,\bar w)$ denotes the expansion of $G_e$ by constants interpreted as the elements represented by $w_1,\dots,w_k$.

\begin{proof}
A finite presentation
\[
G=\gen{a_1,\dots,a_n\mid r_1,\dots,r_m}
\]
is G\"odel-coded by the triple $(n,m,\langle r_1,\dots,r_m\rangle)$. The set of such codes is decidable, so we define an arithmetical predicate $\FPGroupPred(e)$ saying that $e$ encodes a triple $(n,m,R)$ where $R$ is a length-$m$ sequence and every entry of $R$ is a word on $n$ generators.

A word on $n$ generators is coded by the G\"odel number of its letter-sequence. We write $\EltPred(e,w)$ for the decidable predicate saying that $w$ is a word on the generators named in the presentation coded by $e$.

A homomorphism code is a natural number $h$ coding a length-$n$ sequence of word-codes
\[
h=\langle w_1',\dots,w_n'\rangle
\]
in a second presentation $e'$. It represents the homomorphism sending each generator $a_i$ of $e$ to the word $w_i'$ in $G_{e'}$. Let $\TuplePred_n(h)$ be the decidable predicate asserting that $h$ codes a sequence of length $n$, and write $(h)_i$ for its $i$th component. Write $\HomPred(e,h,e')$ for the assertion that $h$ induces a well-defined homomorphism $G_e\to G_{e'}$. Concretely, $\HomPred(e,h,e')$ says that $\TuplePred_n(h)$ holds, that each $(h)_i$ is a word on the generators of $e'$, and that every relator of $e$ maps to the identity in $G_{e'}$. This is a $\Sigma^0_1$ predicate: one searches for derivations witnessing that the image of each relator equals $1$ in $G_{e'}$. Given $h$, the primitive-recursive predicate $\ApplyPred(h,w,w')$ asserts that $w'$ is the image of $w$ under $h$.

Write $u\equiv_e v$ when the word $uv^{-1}$ is derivably trivial from the presentation coded by $e$; this is the arithmetical relation saying that $u$ and $v$ represent the same element of $G_e$. If $\HomPred(e,h,e')$, $u\equiv_e v$, $\ApplyPred(h,u,u')$, and $\ApplyPred(h,v,v')$, then $u'\equiv_{e'}v'$. Indeed, a derivation of $uv^{-1}=1$ in $G_e$ is carried by the coded homomorphism to a derivation of $u'(v')^{-1}=1$ in $G_{e'}$.

For each formula $\phi$ and assignment $\bar w$ we define an arithmetical predicate $\ValPred_{\phi}(e,\bar w)$. For an atomic formula $t_1=t_2$, compute primitive-recursively a word-code
\[
v=v(t_1,t_2,\bar w)
\]
for the word obtained from $t_1t_2^{-1}$ after substituting the assignment $\bar w$ and then performing concatenation and inversion on word-codes. We let $\ValPred_{t_1=t_2}(e,\bar w)$ assert that there exists a derivation of $v=1$ in $G_e$, a $\Sigma^0_1$ condition via the usual semidecision procedure for the word problem. Boolean combinations and first-order quantifiers are handled by the obvious logical clauses, with quantifiers ranging over word-codes satisfying $\EltPred(e,w)$.

It remains to give the recursive modal clauses. For the modal case, write $\bar w=\langle w_1,\dots,w_k\rangle$. Possibility is interpreted by existential quantification over a finitely presented target and a coded homomorphism:
\[
\begin{aligned}
\ValPred_{\possible\psi}(e,\bar w):=\exists e'\exists h\exists w_1'\cdots\exists w_k'\Bigl(&\FPGroupPred(e')\wedge \HomPred(e,h,e')\\
&\wedge \bigwedge_{i=1}^k\bigl(\ApplyPred(h,w_i,w_i')\wedge \EltPred(e',w_i')\bigr)\\
&\wedge \ValPred_{\psi}(e',\langle w_1',\dots,w_k'\rangle)\Bigr).
\end{aligned}
\]
Necessity is interpreted by the corresponding universal condition:
\[
\begin{aligned}
\ValPred_{\necessary\psi}(e,\bar w):=\forall e'\forall h\forall w_1'\cdots\forall w_k'\Bigl(&\FPGroupPred(e')\wedge \HomPred(e,h,e')\\
&\wedge \bigwedge_{i=1}^k\bigl(\ApplyPred(h,w_i,w_i')\wedge \EltPred(e',w_i')\bigr)\\
&\rightarrow \ValPred_{\psi}(e',\langle w_1',\dots,w_k'\rangle)\Bigr).
\end{aligned}
\]
Thus quantification is over both the target presentation $e'$ and the homomorphism code $h$.

For each $\phi(\bar x)$ we set $\phi^{\dagger}(e,\bar w)$ to mean $\ValPred_{\phi}(e,\bar w)$. The mapping $\phi\mapsto \phi^{\dagger}$ is primitive-recursive because the above clauses are.

We shall also use the following representative independence fact. If $\bar u$ and $\bar v$ are tuples of word-codes of the same length with $u_i\equiv_e v_i$ for all $i$, then
\[
\ValPred_{\psi}(e,\bar u)\quad\Longleftrightarrow\quad \ValPred_{\psi}(e,\bar v)
\]
for every modal formula $\psi$ with the corresponding free variables. This is proved by induction on $\psi$. The atomic case follows because term evaluation respects the congruence $\equiv_e$: if corresponding variables name the same elements of $G_e$, then the two evaluated words for any group term are again equivalent modulo $e$. The Boolean and first-order quantifier clauses are immediate. For the modal step, suppose $\ValPred_{\possible\theta}(e,\bar u)$ is witnessed by $e',h,\bar u'$. Choose $\bar v'$ by applying the same coded homomorphism $h$ to $\bar v$. The observation above gives $u_i'\equiv_{e'}v_i'$ for each $i$, and the induction hypothesis for $\theta$ yields $\ValPred_{\theta}(e',\bar v')$. This gives witnesses for $\ValPred_{\possible\theta}(e,\bar v)$; the converse is symmetric. The corresponding statement for boxed formulas follows from the universal modal clause.

We prove by induction on $\phi$ that
\[
(G_e,\bar w)\fpmodels \phi\quad\text{iff}\quad \N\models \phi^{\dagger}(e,\bar w).
\]
The atomic, Boolean, and first-order cases are straightforward, using representative independence to pass from word-codes to the group elements they name. For $\possible\psi$, suppose first that the recursive predicate holds, witnessed by $e'$, $h$, and an image tuple $\bar w'$. These data decode to an actual homomorphism $G_e\to G_{e'}$ witnessing modal satisfaction. Conversely, any witnessing homomorphism can be encoded by choosing words for the images of the generators. Representative-independence ensures that the choice of such words does not affect the truth value of the recursive predicate. The case of $\necessary\psi$ is the corresponding universal argument over all coded target presentations, homomorphism codes, and image tuples.

Finally, if $\phi$ is a sentence, the universal closure over $e$ of
\[
\FPGroupPred(e)\to \phi^{\dagger}(e)
\]
yields a computable reduction from homomorphic modal validity in all finitely presented groups to first-order arithmetic truth. Moreover, if $n=\ulcorner\phi\urcorner$ is the G\"odel code of $\phi$, let
\[
\top_n^{\mathrm{arith}}:=\underbrace{(0=0)\wedge\cdots\wedge(0=0)}_{n+1\text{ conjuncts}}
\]
and define the padded arithmetic sentence
\[
\widehat\phi:=\top_n^{\mathrm{arith}}\wedge \forall e\,\bigl(\FPGroupPred(e)\to \phi^{\dagger}(e)\bigr).
\]
Then $\phi\mapsto \widehat\phi$ is computable and injective on G\"odel codes, while $\widehat\phi$ has the same truth value in arithmetic as the unpadded universal closure. By the totalization convention above, the induced reduction of the homomorphic modal theory of finitely presented groups to true arithmetic may be taken one--one. Consequently, the homomorphic modal theory of finitely presented groups is uniformly interpretable in true arithmetic.
\end{proof}

The converse reduction from true arithmetic requires the arithmetic-interpretation formulas to work not only in the category of all groups and homomorphisms, but also in the category of finitely presented groups and homomorphisms. The following lemmas verify the necessary definability statements in that category.

\begin{lemma}\label{lem:fp-hom-cyclic-membership}
Let $G$ be a finitely presented group and let $x,y\in G$. In the category of finitely presented groups and homomorphisms,
\[
G\fpmodels \necessary\forall t\,(tx=xt\to ty=yt)
\quad\text{iff}\quad y\in\gen{x}\text{ in }G.
\]
\end{lemma}

\begin{proof}
If $y=x^k$, then for every homomorphism from $G$ to a finitely presented group, any element commuting with the image of $x$ also commutes with the image of $y=x^k$. Hence the modal formula holds.

Conversely, suppose the modal formula holds in the category of finitely presented groups and homomorphisms. Fix a finite presentation $G=\gen{a_1,\dots,a_m\mid r_1,\dots,r_\ell}$ and choose a word $w_x$ in the generators representing $x$. Form
\[
G^*=\gen{a_1,\dots,a_m,s\mid r_1,\dots,r_\ell,\; sw_x=w_xs}.
\]
This is finitely presented, and it is the HNN extension obtained by adjoining a stable letter centralizing $\gen{x}$. The canonical embedding $\iota\colon G\hookrightarrow G^*$ is, in particular, a homomorphism between finitely presented groups. Therefore the displayed first-order sentence holds in $G^*$ with $\iota(x),\iota(y)$ in place of $x,y$. Since $s\iota(x)=\iota(x)s$, instantiating $t=s$ gives $s\iota(y)=\iota(y)s$. Lemma~\ref{lem:hnn-centralizer} implies that the elements of $\iota(G)$ commuting with $s$ are exactly $\iota(\gen{x})$, so $\iota(y)\in\iota(\gen{x})$, and hence $y\in\gen{x}$ in $G$.
\end{proof}

\begin{lemma}\label{lem:fp-hom-torsion}
Let $G$ be a finitely presented group and $x\in G$. In the category of finitely presented groups and homomorphisms, the formula $\ord(x)<\infty$ from Theorem~\ref{thm:hom-torsion-element} defines precisely the elements of finite order.
\end{lemma}

\begin{proof}
The proof of Theorem~\ref{thm:hom-torsion-element} uses only the correctness of cyclic subgroup membership and elementary facts about cyclic groups. Lemma~\ref{lem:fp-hom-cyclic-membership} supplies the required cyclic subgroup predicate in the category of finitely presented groups and homomorphisms, so the same argument applies verbatim.
\end{proof}

\begin{lemma}\label{lem:fp-ring-with-parameter}
Let $G$ be a finitely presented group and let $g\in G$ have infinite order. Let $\phi^*(x)$ be the parameter-dependent translation of a ring sentence $\phi$ obtained from Theorem~\ref{thm:presburger-hom} together with Lemma~\ref{lem:multiplication}. When modal operators are interpreted in the category of finitely presented groups and homomorphisms,
\[
G\fpmodels \phi^*[g]
\quad\text{iff}\quad
\langle\Z,+,\cdot,0,1\rangle\models\phi.
\]
\end{lemma}

\begin{proof}
The interpretation represents the integer $i$ by $g^i$. Addition is group multiplication on the cyclic subgroup $\gen{g}$, and divisibility is cyclic subgroup membership. By Lemma~\ref{lem:fp-hom-cyclic-membership}, the cyclic-subgroup predicate has the intended meaning in the category of finitely presented groups and homomorphisms. Thus the quantifier guards range exactly over the interpreted copy of $\Z$, and the divisibility atoms have their intended arithmetic meaning. Lemmas~\ref{lem:squaring} and~\ref{lem:multiplication} are first-order facts about $\langle\Z,+,\mid,0,1\rangle$, so the same induction on formulas as in Theorem~\ref{thm:presburger-hom} proves the displayed equivalence.
\end{proof}

\begin{corollary}\label{cor:fp-ring-translation}
Let $\phi$ be a sentence of the ring language $\{+,\cdot,0,1\}$, and let $\phi^{\dagger}$ be the modal sentence produced in Theorem~\ref{thm:ring-of-integers-hom}. If modal operators are interpreted in the category of finitely presented groups and homomorphisms, then the following are equivalent:
\begin{enumerate}[label=(\arabic*)]
    \item $\langle\Z,+,\cdot,0,1\rangle\models\phi$;
    \item the trivial group satisfies $\phi^{\dagger}$;
    \item every finitely presented group satisfies $\phi^{\dagger}$.
\end{enumerate}
\end{corollary}

\begin{proof}
The parameter-free translation uses the modal definition of infinite order, and Lemma~\ref{lem:fp-ring-with-parameter} gives the correctness of the parameter-dependent ring interpretation in the category of finitely presented groups and homomorphisms. By Lemma~\ref{lem:fp-hom-torsion}, the infinite-order condition has its intended meaning in that category.

If $\langle\Z,+,\cdot,0,1\rangle\models\phi$ and $G$ is finitely presented, then $G\times\Z$ is finitely presented and the map $G\to G\times\Z$ given by $g\mapsto(g,0)$ witnesses $G\fpmodels\phi^{\dagger}$, using the element $(e_G,1)$ as the infinite-order parameter. Hence (1) implies (3), and (3) implies (2) because the trivial group is finitely presented.

Conversely, if the trivial group satisfies $\phi^{\dagger}$ in the category of finitely presented groups and homomorphisms, then some finitely presented group $H$ with an infinite-order element $h$ satisfies the parameter-dependent translation $\phi^*(h)$. Lemma~\ref{lem:fp-ring-with-parameter} then gives $\langle\Z,+,\cdot,0,1\rangle\models\phi$. Thus (2) implies (1).
\end{proof}

\begin{theorem}\label{thm:fp-hom-isomorphic}
The homomorphic modal theory of finitely presented groups is computably isomorphic to the theory of true arithmetic.
\end{theorem}

\begin{proof}
By Myhill's isomorphism theorem \cite{Myhill55}, it suffices to show that:
\begin{enumerate}[label=(\arabic*)]
    \item the homomorphic modal theory of finitely presented groups is one--one reducible to true arithmetic, and
    \item true arithmetic is one--one reducible to the homomorphic modal theory of finitely presented groups.
\end{enumerate}
For statement~(1), let $\sigma$ be a modal sentence and let $n=\ulcorner\sigma\urcorner$ be its G\"odel code. Write
\[
\top_n^{\mathrm{arith}}:=\underbrace{(0=0)\wedge\cdots\wedge(0=0)}_{n+1\text{ conjuncts}}.
\]
By the final paragraph of the proof of Theorem~\ref{thm:fp-hom}, the map
\[
\sigma\longmapsto \top_n^{\mathrm{arith}}\wedge \forall e\,\bigl(\FPGroupPred(e)\to \sigma^{\dagger}(e)\bigr)
\]
is computable and injective, and its image is arithmetically true exactly when $\sigma$ is valid in all finitely presented groups in the homomorphism semantics. Thus the homomorphic modal theory of finitely presented groups is one--one reducible to true arithmetic.

For statement~(2), let $\sigma$ be a first-order arithmetic sentence and let $n=\ulcorner\sigma\urcorner$. Write
\[
\top_n^{\mathrm{grp}}:=\underbrace{(e=e)\wedge\cdots\wedge(e=e)}_{n+1\text{ conjuncts}}
\]
and put
\[
\widetilde\sigma:=\top_n^{\mathrm{grp}}\wedge (\sigma^{\Z})^{\dagger},
\]
where $\sigma^{\Z}$ is the ring sentence from Lemma~\ref{lem:true-arithmetic-to-Z-hom} and $(\sigma^{\Z})^{\dagger}$ is the translation from Theorem~\ref{thm:ring-of-integers-hom}. The map $\sigma\mapsto \widetilde\sigma$ is computable and injective on G\"odel codes. Moreover,
\[
\begin{aligned}
\langle \N,+,\cdot,0,1\rangle\models \sigma
&\quad\text{iff}\quad
\langle \Z,+,\cdot,0,1\rangle\models \sigma^{\Z}\\
&\quad\text{iff}\quad
(\sigma^{\Z})^{\dagger}\text{ is valid in all finitely presented groups}\\
&\quad\text{iff}\quad
\widetilde\sigma\text{ is valid in all finitely presented groups,}
\end{aligned}
\]
where the second equivalence is Corollary~\ref{cor:fp-ring-translation}, and the last equivalence uses that $\top_n^{\mathrm{grp}}$ is valid in every group. By the totalization convention, true arithmetic is one--one reducible to the homomorphic modal theory of finitely presented groups.

Myhill's theorem now yields the claimed computable isomorphism.
\end{proof}

\subsection{Countable groups}

Throughout this subsection, modal truth is interpreted in the potentialist system of countable groups and homomorphisms between countable groups. Accordingly, if $G$ codes a countable group and $\bar a\in D_G^k$, then
\[
(G,\bar a)\ctblhommodels\varphi
\]
means that $\varphi$ holds at the coded group under the assignment $\bar a$ when the modal operator $\possible$ quantifies only over homomorphisms into coded countable groups.

We code a countable group by a single set $G\subseteq\N$ using tagged pairs, matching the convention in the companion paper. Fix a primitive-recursive pairing function $\langle\cdot,\cdot\rangle$, and write
\[
\begin{aligned}
D_G(x) &\quad\text{if and only if}\quad \langle 0,x\rangle\in G,\\
M_G(x,y,z) &\quad\text{if and only if}\quad \langle 1,\langle x,\langle y,z\rangle\rangle\rangle\in G.
\end{aligned}
\]
Thus $D_G$ is the coded domain and $M_G$ is the coded multiplication graph.

\begin{theorem}\label{thm:countable-hom}
For every countable group-code $G\subseteq\N$, every tuple $\bar a$ of natural numbers from the coded domain $D_G$, and every formula $\phi(\bar x)$ of the modal language of groups with $\possible,\necessary$ interpreted via homomorphisms between countable groups, there is a primitive-recursive second-order arithmetic formula $\phi^{\ddagger}(G,\bar a)$ such that
\[
(G,\bar a)\ctblhommodels \phi
\quad\text{if and only if}\quad
\langle \N,+,\cdot,0,1,\Pow{\N}\rangle\satisfies \phi^{\ddagger}(G,\bar a).
\]
Consequently, the homomorphic modal theory of countable groups under homomorphisms between countable groups is computably interpretable in true second-order arithmetic.
\end{theorem}

\begin{proof}
Any countable group $K$ can be coded in this way: choose an injection $i\colon K\to\N$ with $i(1_K)=0$, let $D_G$ be the range of $i$, and let $M_G(i(a),i(b),i(c))$ hold exactly when $ab=c$ in $K$. Conversely, every set $G\subseteq\N$ satisfying the corresponding second-order axioms determines a countable group with domain $D_G$ and multiplication $M_G$.

A second-order formula $\GroupPred(G)$ asserts that $0\in D_G$, that $M_G$ is total and functional on $D_G$, that no multiplication triple uses elements outside $D_G$, and that multiplication satisfies the group axioms with identity $0$ and inverses in $D_G$. Explicitly, the no-garbage clause is
\[
\forall x\forall y\forall z\bigl(M_G(x,y,z)\to D_G(x)\wedge D_G(y)\wedge D_G(z)\bigr),
\]
and totality and functionality are expressed by
\[
\forall x\forall y\bigl(D_G(x)\wedge D_G(y)\to \exists!z\,(D_G(z)\wedge M_G(x,y,z))\bigr).
\]

Given group-codes $G,H\subseteq\N$, a homomorphism from the coded group of $G$ to the coded group of $H$ is coded by a set $F\subseteq\N$, where $\langle x,y\rangle\in F$ means $f(x)=y$. A second-order formula $\HomPred(G,F,H)$ says that $F$ is the graph of a total homomorphism from $D_G$ into $D_H$. Thus $F$ has no graph-theoretic garbage,
\[
\forall x\forall y\bigl(\langle x,y\rangle\in F\to D_G(x)\wedge D_H(y)\bigr),
\]
it is total and functional on $D_G$,
\[
\forall x\bigl(D_G(x)\to \exists!y\,(D_H(y)\wedge \langle x,y\rangle\in F)\bigr),
\]
and it preserves multiplication:
\[
\begin{aligned}
\forall x\forall y\forall z\forall x'\forall y'\forall z'\Bigl(&D_G(x)\wedge D_G(y)\wedge D_G(z)\wedge M_G(x,y,z)\\
&\wedge\langle x,x'\rangle\in F\wedge\langle y,y'\rangle\in F\wedge\langle z,z'\rangle\in F\\
&\to M_H(x',y',z')\Bigr).
\end{aligned}
\]

Before giving the recursive translation, recall that the formal group language is inverse-free. Thus every group term is first rewritten into the language $\{\cdot,e\}$, and then the recursion below is applied. For each group term $t(\bar x)$ we define, by recursion on term complexity, a second-order arithmetic formula $\operatorname{Eval}_{t}(G,\bar x,y)$ asserting that $y$ is the value of $t$ in the group coded by $G$ under the assignment $\bar x\in D_G^k$:
\[
\operatorname{Eval}_{x_i}(G,\bar x,y):=(D_G(y)\wedge y=x_i),\qquad
\operatorname{Eval}_{e}(G,\bar x,y):=(y=0),
\]
and
\[
\operatorname{Eval}_{u\cdot v}(G,\bar x,y):=\exists y_0\exists y_1\bigl(\operatorname{Eval}_{u}(G,\bar x,y_0)\wedge \operatorname{Eval}_{v}(G,\bar x,y_1)\wedge M_G(y_0,y_1,y)\bigr).
\]
The map $t\mapsto\operatorname{Eval}_{t}$ is primitive-recursive on \Godel\ codes of terms.

For each modal formula $\psi(\bar x)$ we now define, by structural induction, a second-order arithmetic formula $\ValPred_{\psi}(G,\bar x)$. For an atomic equation $t_1=t_2$, let
\[
\ValPred_{t_1=t_2}(G,\bar x):=\exists y_1\exists y_2\bigl(\operatorname{Eval}_{t_1}(G,\bar x,y_1)\wedge \operatorname{Eval}_{t_2}(G,\bar x,y_2)\wedge y_1=y_2\bigr).
\]
Boolean connectives are translated directly, and first-order quantifiers are relativized to the coded domain:
\[
\ValPred_{\exists u\,\psi}(G,\bar x):=\exists u\bigl(D_G(u)\wedge \ValPred_{\psi}(G,\bar x,u)\bigr),
\]
\[
\ValPred_{\forall u\,\psi}(G,\bar x):=\forall u\bigl(D_G(u)\to \ValPred_{\psi}(G,\bar x,u)\bigr).
\]
Here the right-hand sides use the extended assignment tuple obtained by appending $u$ to $\bar x$.

For the modal clauses, suppose the free variables of $\psi$ are among $\bar x=(x_1,\dots,x_k)$, and put $\bar y=(y_1,\dots,y_k)$. Define
\[
\begin{aligned}
\ValPred_{\possible\psi}(G,\bar x):={}&\exists H\exists F\exists y_1\cdots\exists y_k\Bigl(\GroupPred(H)\wedge \HomPred(G,F,H)\\
&\wedge \bigwedge_{i=1}^k \langle x_i,y_i\rangle\in F\wedge \ValPred_{\psi}(H,\bar y)\Bigr),
\end{aligned}
\]
and
\[
\begin{aligned}
\ValPred_{\necessary\psi}(G,\bar x):={}&\forall H\forall F\forall y_1\cdots\forall y_k\Bigl((\GroupPred(H)\wedge \HomPred(G,F,H)\\
&\wedge \bigwedge_{i=1}^k \langle x_i,y_i\rangle\in F)\to \ValPred_{\psi}(H,\bar y)\Bigr).
\end{aligned}
\]
Because $\HomPred(G,F,H)$ says that $F$ is the graph of a total homomorphism from $D_G$ into $D_H$, the tuple $\bar y$ is exactly the image tuple of $\bar x$. Every clause in the recursive definition is primitive-recursive on \Godel\ codes, so $\psi\mapsto\ValPred_{\psi}$ is primitive-recursive.

Set $\phi^{\ddagger}(G,\bar a):=\ValPred_{\phi}(G,\bar a)$. An induction on group terms shows that $\operatorname{Eval}_{t}(G,\bar a,b)$ holds exactly when $b$ is the value of $t$ in the countable group coded by $G$ under the assignment $\bar a\in D_G^k$. Using this, a second induction on the complexity of $\phi$ yields the desired equivalence. The atomic and Boolean steps are immediate. The quantifier step uses the relativization to $D_G$. In the modal step, witnesses $H,F,\bar y$ decode to an actual homomorphism of coded countable groups carrying $\bar a$ to its image tuple, and conversely any witnessing homomorphism in the category of countable groups and homomorphisms yields suitable second-order parameters.

If $\phi$ is a sentence, define
\[
\widehat\phi:=\forall G\,(\GroupPred(G)\to \phi^{\ddagger}(G)).
\]
Then $\phi$ is valid in all countable groups under homomorphisms if and only if
\[
\langle\N,+,\cdot,0,1,\Pow{\N}\rangle\satisfies\widehat\phi.
\]
To obtain a one--one reduction, fix a computable family $(\eta_n)_{n\in\N}$ of pairwise syntactically distinct true second-order arithmetic sentences; for definiteness, let $\eta_n$ be the right-associated conjunction of $n+1$ copies of the true atomic sentence $0=0$. Define
\[
\widetilde{\widehat\phi}:=\eta_{\operatorname{Code}(\phi)}\wedge \widehat\phi.
\]
Since each $\eta_n$ is true, this preserves truth. The map $\phi\mapsto\widetilde{\widehat\phi}$ is primitive-recursive and injective on \Godel\ codes of sentences because the left conjunct records the code of $\phi$ inside the fixed conjunction template. By the totalization convention above, this sentence map induces a total one--one reduction of modal validity in the category of countable groups and homomorphisms to truth in second-order arithmetic.
\end{proof}

\begin{corollary}\label{cor:countable-hom-soa}
The homomorphic modal theory of countable groups under homomorphisms between countable groups is one--one reducible to true second-order arithmetic.
\end{corollary}

\begin{proof}
This is exactly the one--one reduction induced, by the totalization convention above, by the injective sentence map $\phi\mapsto\widetilde{\widehat\phi}$ constructed at the end of the proof of Theorem~\ref{thm:countable-hom}.
\end{proof}

\section{Propositional modal validities of groups}\label{section:validities}

We now compute the modal validities of groups under arbitrary homomorphisms. For modal validities under embeddings, and for historical context concerning the question of whether all groups validate $\SFourTwo$, see the companion embedding paper and the earlier work of Berger, Block, and L\"owe on abelian groups \cite{WoloszynEmbedding,BBL23}.

For this section, a \emph{substitution language} $\mathcal L$ means a nonempty class of formulas with $\Lgrp\subseteq\mathcal L\subseteq\LgrpModal$, closed under Boolean connectives. I write $\Val_{\GrpHom}(G,\mathcal L)$ for sentential, no-parameter substitutions, and $\Val_{\GrpHom}(G,\mathcal L_A)$ when parameters from $A\subseteq G$ are allowed and transported along homomorphisms.

If $A\subseteq G$, I say that homomorphisms out of $G$ are \emph{$A$-amalgamable} if, whenever we have homomorphisms
\[
G\xrightarrow{\,f\,}H,
\qquad
K_0\xleftarrow{\,f_0\,}H\xrightarrow{\,f_1\,}K_1,
\]
there are a group $P$ and homomorphisms $g_0\colon K_0\to P$ and $g_1\colon K_1\to P$ such that
\[
(g_0\circ f_0\circ f)(a)=(g_1\circ f_1\circ f)(a)
\]
for every $a\in A$. Thus the two continuations can be amalgamated over the chosen parameters.

\begin{definition}\label{def:buttons-dials-hom}
Let $G$ be a group.
\begin{enumerate}[label=(\arabic*)]
    \item A statement $b$ is a \emph{button} at $G$ if $G\hommodels \necessary\possible\necessary b$. It is \emph{pushed} if already $G\hommodels\necessary b$.
    \item A button $b$ at $G$ is \emph{pure} if its truth is persistent: whenever $b$ holds after a homomorphism out of $G$, it continues to hold after all further homomorphisms.
    \item A finite list $d_0,\dots,d_{n-1}$ is a \emph{dial} at $G$ if necessarily exactly one $d_i$ holds, and for every homomorphism $h\colon G\to H$ each dial value can be realized by some further homomorphism $k\colon H\to K$.
    \item A finite family of buttons is \emph{independent} at $G$ if, after every homomorphism $h\colon G\to H$ out of $G$, any chosen subfamily of the buttons not yet pushed at $H$ can be pushed by a further homomorphism out of $H$ without affecting the status of the others. A dial is \emph{independent} of a family of buttons if, from every such homomorphism, its value can be changed arbitrarily by a further homomorphism without changing which of those buttons are pushed.
\end{enumerate}
\end{definition}

The following proposition collects the standard lower-bound and upper-bound tools for modal model theory; see Hamkins--Woloszyn~\cite[Sections~2 and~5]{HW24} and the categorical exposition for sets~\cite{WSet}. In upper-bound applications $\mathcal L_A$ is required to contain the relevant control statements.

\begin{proposition}\label{prop:general-tools}
Let $G$ be a group, let $A\subseteq G$, and let $\mathcal L$ be a substitution language.
\begin{enumerate}[label=(\arabic*)]
    \item If homomorphisms out of $G$ are $A$-amalgamable, then
    \[
    \SFourTwo\subseteq\Val_{\GrpHom}(G,\mathcal L_A).
    \]
    \item If $G$ admits arbitrarily long finite dials and $\mathcal L_A$ contains the corresponding dial statements, then $\Val_{\GrpHom}(G,\mathcal L_A)\subseteq \SFive$.
    \item $\SFive\subseteq \Val_{\GrpHom}(G,\mathcal L)$ for sentential substitutions.
    \item For the trivial group $\{e\}$, $\SFive\subseteq \Val_{\GrpHom}(\{e\},\mathcal L_{\{e\}})$.
    \item If $G$ admits arbitrarily large finite independent families of buttons together with arbitrarily long finite dials independent of those buttons, and $\mathcal L_A$ contains the corresponding button and dial statements, then $\Val_{\GrpHom}(G,\mathcal L_A)\subseteq \SFourTwo$.
\end{enumerate}
\end{proposition}

\begin{proof}
The lower-bound and upper-bound statements are the standard control-statement theorems cited above. For (3), suppose $G\hommodels \possible\necessary\varphi$, witnessed by a homomorphism $f\colon G\to H$ with $H\hommodels \necessary\varphi$. There is always a homomorphism $r\colon H\to G$, namely the trivial homomorphism. Because $H\hommodels \necessary\varphi$, applying necessity along $r$ gives $G\hommodels\varphi$. Thus $G$ validates the axiom $5$ for sentential substitutions, and therefore validates $\SFive$.

For (4), suppose $\{e\}\hommodels \possible\necessary\varphi[e]$, witnessed by the unique homomorphism $f\colon \{e\}\to H$ with $H\hommodels \necessary\varphi[f(e)]$. There is a unique homomorphism $r\colon H\to\{e\}$, and $r\circ f=\mathrm{id}_{\{e\}}$. Applying $\necessary\varphi$ along $r$, we obtain $\{e\}\hommodels\varphi[e]$. Hence the trivial group validates axiom $5$ with its only parameter, and therefore validates $\SFive$.
\end{proof}

\subsection{Uniform lower and upper bounds}

\begin{theorem}\label{thm:hom-s42-lower}
For every group $G$, every substitution language $\mathcal L$ of group-modal formulas with
\[
\Lgrp\subseteq \mathcal L\subseteq \LgrpM,
\]
and every parameter set $A\subseteq G$,
\[
\SFourTwo\subseteq \Val_{\GrpHom}(G,\mathcal L_A).
\]
\end{theorem}

\begin{proof}
Let $A$ be the chosen parameter set. We show that homomorphisms out of $G$ are $A$-amalgamable. Let $f\colon G\to H$ be a homomorphism, and let
\[
K_0 \xleftarrow{f_0} H \xrightarrow{f_1} K_1
\]
be a span of group homomorphisms rooted at $H$. Form the free product $K_0*K_1$ and quotient by the normal closure of all relations
\[
f_0(h)=f_1(h)\qquad(h\in H).
\]
Denote the quotient by $P$, and let $j_0\colon K_0\to P$ and $j_1\colon K_1\to P$ be the canonical homomorphisms. Then
\[
j_0\circ f_0=j_1\circ f_1
\]
on all of $H$, and therefore
\[
(j_0\circ f_0\circ f)(a)=(j_1\circ f_1\circ f)(a)
\]
for every $a\in A$. Thus homomorphisms out of $G$ are $A$-amalgamable. Proposition~\ref{prop:general-tools}(1) yields the result.
\end{proof}

\begin{theorem}\label{thm:hom-s5-upper}
For every group $G$ and every substitution language $\mathcal L$ of group-modal formulas with
\[
\Lgrp\subseteq \mathcal L\subseteq \LgrpM,
\]
we have
\[
\Val_{\GrpHom}(G,\mathcal L)\subseteq \SFive.
\]
In particular, the sentential modal validities of every group are contained in $\SFive$.
\end{theorem}

\begin{proof}
We exhibit arbitrarily long finite dials. Fix $N\geq 1$. For $1\leq n<N$, let $d_n$ be the sentence asserting that the group has size exactly $n$, and let $d_{\geq N}$ assert that the group has size at least $N$. Exactly one of
\[
d_1,\dots,d_{N-1},d_{\geq N}
\]
holds in any group.

Now let $H$ be any group. For each $1\leq n<N$ there exists a group of size exactly $n$, for instance the cyclic group $C_n$, and there exists a group of size at least $N$, for instance $C_N$ or $\Z$. Since for any groups $H$ and $K$ there is a homomorphism $H\to K$ (the trivial homomorphism), every dial value is reachable from $H$. Thus the displayed sentences form an $N$-dial at every group. By Proposition~\ref{prop:general-tools}(2), applied with parameter set $A=\varnothing$, the sentential validities are contained in $\SFive$.
\end{proof}

\begin{theorem}\label{thm:hom-s5-sentential}
For every group $G$, the propositional modal validities with respect to sentential substitution instances from any substitution language $\mathcal L$ of group-modal formulas with
\[
\Lgrp\subseteq \mathcal L\subseteq \LgrpM
\]
are exactly $\SFive$.
\end{theorem}

\begin{proof}
For every group $H$ there exists a homomorphism $H\to G$, namely the trivial homomorphism. Hence Proposition~\ref{prop:general-tools}(3) yields the lower bound $\SFive$ for sentential substitutions. The opposite inclusion is Theorem~\ref{thm:hom-s5-upper}.
\end{proof}

\begin{remark}\label{rem:nontrivial-not-s5}
If $G$ is nontrivial and $a\in G$ is nonidentity, then the formula $a=e$ is an unpushed button at $G$. Indeed, $a=e$ is false in $G$, but one may push it by applying the quotient map to $G/\gen{a}^{\normalfont G}$, or more simply by using the trivial homomorphism from $G$ to the trivial group. Consequently no nontrivial group has parameter-validities exactly $\SFive$.
\end{remark}

\begin{theorem}\label{thm:hom-trivial-s5}
Let $\{e\}$ be the trivial group. For every substitution language $\mathcal L$ of group-modal formulas with
\[
\Lgrp\subseteq \mathcal L\subseteq \LgrpM,
\]
we have
\[
\Val_{\GrpHom}(\{e\},\mathcal L_{\{e\}})=\SFive.
\]
\end{theorem}

\begin{proof}
Proposition~\ref{prop:general-tools}(4) yields the lower bound $\SFive$ for the trivial group. For the upper bound, observe that the only parameter from the trivial group is its identity element, which is already named by the group-language constant $e$. Thus allowing parameters from $\{e\}$ adds no new substitution instances beyond the sentential language $\mathcal L$, and Theorem~\ref{thm:hom-s5-upper} yields
\[
\Val_{\GrpHom}(\{e\},\mathcal L_{\{e\}})\subseteq \SFive.
\]
\end{proof}

\subsection{The exact formulaic modal theory}

To produce exact $\SFourTwo$ validities with parameters, we use divisibility buttons and center-size dials.

\begin{definition}
Let $p$ be a prime.
\begin{enumerate}[label=(\arabic*)]
    \item An element $x$ of a group is \emph{$p$-divisible} if there exists $y$ with $y^p=x$.
    \item An element $x$ is \emph{$p$-indivisible} if it is not $p$-divisible.
    \item A group $G$ is \emph{uniformly prime-indivisible} if for every finite set $S$ of primes there exists an element $a\in G$ of infinite order such that $a$ is $p$-indivisible in $G$ for every $p\in S$.
\end{enumerate}
\end{definition}

We shall use the following standard notation for amalgamated free products. Let $G$ and $A$ be groups and let $C$ be a group with injective homomorphisms $\iota_G\colon C\hookrightarrow G$ and $\iota_A\colon C\hookrightarrow A$. We view the amalgamated free product $H=G*_C A$ as the quotient of the free product $G*A$ by the normal closure of the relations $\iota_G(c)=\iota_A(c)$ for all $c\in C$. Via the injections we identify $C$ with its images in $G$ and $A$ and regard $G,A,C$ as subgroups of $H$.

Recall that a word $x_1x_2\cdots x_n$ with each $x_i\in G\cup A$ is \emph{reduced} if:
\begin{enumerate}[label=(\arabic*)]
    \item $x_i\neq 1$ for all $i$,
    \item consecutive syllables lie in different factors, and
    \item if $n\geq 2$, then $x_i\notin C$ for all $i$.
\end{enumerate}
The syllable length of such a word is $n$. Under the injectivity hypothesis, every element of $H=G*_C A$ admits a reduced representative, and if a reduced word has syllable length at least $2$, then it represents an element of $H$ lying in neither $G$ nor $A$.

\begin{lemma}\label{lem:central-conjugacy}
Suppose that $C$ is central in $A$. Let $c,c'\in C$. If there exists $h\in H=G*_C A$ such that
\[
h^{-1}ch=c',
\]
then there exists $g\in G$ such that
\[
g^{-1}cg=c'.
\]
\end{lemma}

\begin{proof}
If $h=1$, then $c'=c$ and we may take $g=1\in G$. So assume $h\neq 1$, and choose a reduced representative
\[
h=x_1x_2\cdots x_n
\]
with $n\geq 1$.

For $0\leq i\leq n$ define
\[
c_0=c,\qquad c_i=(x_1\cdots x_i)^{-1}c(x_1\cdots x_i).
\]
Then $c_n=h^{-1}ch=c'\in C$. For each $m\in\{1,\dots,n\}$ set
\[
s_m=x_mx_{m+1}\cdots x_n,
\]
so that $h=(x_1\cdots x_{m-1})s_m$ and hence
\[
c_n=s_m^{-1}c_{m-1}s_m.
\]
We prove by induction on $m$ that $c_{m-1}\in C$ (and hence also $c_m\in C$). The base case $m=1$ is clear. Assume $c_{m-1}\in C$. Write $s_m=y_1\cdots y_r$ with $y_1=x_m$, and set
\[
d=y_1^{-1}c_{m-1}y_1=c_m.
\]
If $r=1$, then $s_m=y_1$ and $d=s_m^{-1}c_{m-1}s_m=c_n\in C$, so $c_m\in C$. Suppose $r\geq 2$ and $d\notin C$. Let $v=y_2\cdots y_r$. Then
\[
s_m^{-1}c_{m-1}s_m=v^{-1}dv.
\]
We claim that $v^{-1}dv$ is reduced. Since $s_m$ is reduced of length at least $2$, we have $y_2,\dots,y_r\notin C$, hence $y_i^{\pm1}\notin C$. By assumption $d\notin C$. Because $s_m$ is reduced, $y_2$ lies in the factor opposite to $y_1$. Since $c_{m-1}\in C\subseteq G\cap A$, the element $d=y_1^{-1}c_{m-1}y_1$ lies in the same factor as $y_1$, so $y_2$ and $y_2^{-1}$ lie in the factor opposite to $d$. Therefore no reduction occurs at the junctions $y_2^{-1}d$ and $dy_2$, and elsewhere reducedness comes from the reducedness of $v$ and $v^{-1}$. Hence $v^{-1}dv$ is reduced of syllable length $2r-1\geq 3$, so by the normal form theorem it cannot represent an element of $C$, contrary to $v^{-1}dv=c_n\in C$. Thus $d\in C$, and the induction is complete.

Now observe that if $x_i\in A$, then $c_i=x_i^{-1}c_{i-1}x_i=c_{i-1}$ because $C$ is central in $A$. Thus only the $G$-syllables affect the conjugation chain. Let $i_1<\cdots<i_k$ be the indices with $x_{i_j}\in G$, and set
\[
g=x_{i_1}x_{i_2}\cdots x_{i_k}\in G,
\]
with $g=1$ if there are no such indices. Defining the corresponding $G$-only conjugation chain shows by induction on $i$ that the original $c_i$ coincide with the conjugates obtained by successively applying only the $G$-syllables. In particular,
\[
c_n=g^{-1}cg.
\]
Since $c_n=c'$, we obtain $g^{-1}cg=c'$ with $g\in G$.
\end{proof}

\begin{lemma}\label{lem:one-prime-step}
Let $p$ be a prime, let $G$ be a group, and let $a\in G$ be an element of infinite order. Set
\[
A_p=\Z\left[\tfrac1p\right]\times \Z
\]
and form the amalgamated free product
\[
H=G*_{\gen{a}}A_p
\]
via the identification $a^n\leftrightarrow (n,0)$. Then $a$ becomes $p$-divisible in $H$. Moreover, for every prime $q\neq p$ such that $a$ is $q$-indivisible in $G$, the element $a$ remains $q$-indivisible in $H$.
\end{lemma}

\begin{proof}
We first show that $a$ becomes $p$-divisible in $H$. In $A_p$ one has
\[
p\left(\tfrac1p,0\right)=(1,0),
\]
and $(1,0)$ is identified with $a$. Thus $\left(\tfrac1p,0\right)$ is a $p$th root of $a$ in $H$.

Now let $q\neq p$ be prime, and assume that $a$ is $q$-indivisible in $G$. We claim that $a$ is still $q$-indivisible in $H$. Suppose for a contradiction that there exists $x\in H$ with
\[
x^q=a.
\]
Let $T$ be the Bass--Serre tree associated with the splitting $H=G*_{\gen{a}}A_p$. We use the standard facts about the Bass--Serre action of $H$ on $T$: the action is without inversions, every element is either elliptic or hyperbolic, powers of a hyperbolic element are hyperbolic, and the fixed-point set of an element is a subtree; see \cite[Chapter~I, \S\S4--6]{SerreTrees}.

The vertices of $T$ are the left cosets $H/G$ and $H/A_p$, the edges are the left cosets $H/\gen{a}$, and an edge $h\gen{a}$ joins the vertices $hG$ and $hA_p$. Let $v_G=G$ and $v_A=A_p$ be the base vertices, and let $e=\gen{a}$ be the base edge joining them. The stabilizers are
\[
\mathrm{Stab}(h\cdot v_G)=hGh^{-1},\qquad \mathrm{Stab}(h\cdot v_A)=hA_ph^{-1},\qquad \mathrm{Stab}(h\cdot e)=h\gen{a}h^{-1}.
\]
In particular, $a\in\gen{a}=\mathrm{Stab}(e)$, so $a$ fixes $e$. Since $x^q=a$, the element $x^q$ fixes $e$, hence is elliptic. If $x$ were hyperbolic, then every nonzero power of $x$ would also be hyperbolic, impossible. So $x$ is elliptic and fixes some vertex $v$ of $T$.

Choose $h\in H$ such that $h\cdot v$ is one of the base vertices $v_G$ or $v_A$, and set
\[
y=hxh^{-1}.
\]
Then $y$ fixes $h\cdot v$, so either $y\in G$ or $y\in A_p$. Moreover,
\[
y^q=hx^qh^{-1}=hah^{-1}.
\]
Since $hah^{-1}\in h\gen{a}h^{-1}=\mathrm{Stab}(h\cdot e)$, the element $y^q$ fixes the edge $h\cdot e$, and hence both endpoints of that edge.

Assume first that $y\in A_p$. Then $y^q\in A_p$, so $y^q$ fixes the base vertex $v_A$. We have already observed that $y^q$ fixes the vertex $h\cdot v_G$ as well. The fixed-point set of $y^q$ is a subtree, so it contains the geodesic from $v_A$ to $h\cdot v_G$. Let $f$ be the first edge on that geodesic issuing from $v_A$. Then $y^q$ fixes $f$. Any edge incident to $v_A$ has the form $s\gen{a}$ with $s\in A_p$, and such an edge is $s\cdot e$. Because $A_p$ is abelian, $\gen{a}$ is central in $A_p$, so
\[
\mathrm{Stab}(f)=s\gen{a}s^{-1}=\gen{a}.
\]
Hence $y^q\in\gen{a}$, say $y^q=a^m$. Write $y=(r,n)\in A_p$ in additive notation. Then
\[
q(r,n)=(m,0),
\]
so $qn=0$ and $qr=m$. Since $\Z$ is torsion-free, $n=0$, and therefore $r=m/q$. But $r\in \Z[1/p]$ and $q\neq p$, so $q\mid m$, say $m=qk$. Thus $y=(k,0)=a^k\in\gen{a}$. We have
\[
hah^{-1}=y^q=a^{qk}.
\]
So $a$ and $a^{qk}$ are conjugate in $H$ and both lie in $\gen{a}$. By Lemma~\ref{lem:central-conjugacy}, there exists $g\in G$ such that
\[
g^{-1}ag=a^{qk}.
\]
Equivalently,
\[
(ga^kg^{-1})^q=a,
\]
contradicting the assumption that $a$ is $q$-indivisible in $G$.

Assume now that $y\in G$. Then $y^q\in G$, so $y^q$ fixes $v_G$. As before, $y^q$ also fixes $h\cdot v_A$, so the fixed-point set of $y^q$ contains the geodesic from $v_G$ to $h\cdot v_A$. Let $f$ be the first edge on that geodesic issuing from $v_G$. Then $y^q$ fixes $f$. Any edge incident to $v_G$ has the form $g_0\gen{a}$ with $g_0\in G$, so
\[
\mathrm{Stab}(f)=g_0\gen{a}g_0^{-1}
\]
for some $g_0\in G$. Hence $y^q\in g_0\gen{a}g_0^{-1}$. Setting $y_1=g_0^{-1}yg_0\in G$, we obtain
\[
y_1^q\in\gen{a},\qquad y_1^q=a^m
\]
for some $m\in\Z$. Moreover,
\[
a^m=y_1^q=g_0^{-1}y^qg_0=g_0^{-1}hah^{-1}g_0,
\]
so again $a$ and $a^m$ are conjugate in $H$ and both lie in $\gen{a}$. Applying Lemma~\ref{lem:central-conjugacy} once more, there exists $g\in G$ with
\[
g^{-1}ag=a^m.
\]
Equivalently,
\[
(gy_1g^{-1})^q=a,
\]
again contradicting the $q$-indivisibility of $a$ in $G$.

Both cases lead to contradictions. Therefore no $x\in H$ satisfies $x^q=a$, and $a$ remains $q$-indivisible in $H$ for every prime $q\neq p$ such that $a$ is $q$-indivisible in $G$.
\end{proof}

\begin{lemma}\label{lem:finite-order-divisibility}
Let $a$ be an element of finite order $m$ in a group $G$, and let $p$ be a prime. If $p\nmid m$, then $a$ is $p$-divisible already in the cyclic subgroup $\gen{a}$. In particular, if $a$ is $p$-indivisible in $G$, then either $a$ has infinite order or $p\mid \ord(a)$.
\end{lemma}

\begin{proof}
The $p$th-power map is an automorphism of the finite cyclic group $\gen{a}\cong C_m$ whenever $\gcd(p,m)=1$. Hence, if $p\nmid m$, the element $a$ already has a $p$th root in $\gen{a}$.
\end{proof}

\begin{lemma}\label{lem:one-prime-step-finite}
Let $p$ be a prime, let $G$ be a group, and let $a\in G$ be an element of finite order $m$ with $p\mid m$. Set
\[
A_{p,m}=C_{pm}=\gen{t\mid t^{pm}=1}
\]
and form the amalgamated free product
\[
H=G*_{\gen{a}}A_{p,m}
\]
via the identification $a=t^p$. Then $a$ becomes $p$-divisible in $H$. Moreover, for every prime $q\neq p$ such that $a$ is $q$-indivisible in $G$, the element $a$ remains $q$-indivisible in $H$.
\end{lemma}

\begin{proof}
The element $t\in A_{p,m}\leq H$ is a $p$th root of $a$, so $a$ becomes $p$-divisible in $H$.

Now let $q\neq p$ be a prime, and assume that $a$ is $q$-indivisible in $G$. Suppose towards a contradiction that there exists $x\in H$ with
\[
x^q=a.
\]
Let $T$ be the Bass--Serre tree associated with the splitting $H=G*_{\gen{a}}A_{p,m}$. As in the proof of Lemma~\ref{lem:one-prime-step}, the element $x$ must be elliptic, because $x^q=a$ fixes the base edge. Choose $h\in H$ so that $y=hxh^{-1}$ fixes one of the base vertices. Then $y\in G$ or $y\in A_{p,m}$, and
\[
y^q=hah^{-1}.
\]
Moreover, $y^q$ fixes the edge $h\cdot e$.

Assume first that $y\in A_{p,m}$. Since $A_{p,m}$ is abelian and $\gen{a}$ is central in it, the same subtree argument as in Lemma~\ref{lem:one-prime-step} shows that $y^q\in \gen{a}$. Write $A_{p,m}=\gen{t}$ and $y=t^r$. Since $y^q\in \gen{t^p}$, there is some integer $k$ with
\[
t^{rq}=t^{pk}.
\]
Thus $rq\equiv pk\pmod{pm}$. Reducing modulo $p$ gives $rq\equiv 0\pmod p$, and since $q\neq p$ we conclude that $p\mid r$. Hence $y\in \gen{t^p}=\gen{a}\subseteq G$. Because $y^q=hah^{-1}$ and both $a$ and $y^q$ lie in $\gen{a}$, Lemma~\ref{lem:central-conjugacy} yields some $g\in G$ with
\[
g^{-1}ag=y^q.
\]
Conjugating by $g$ gives
\[
(gyg^{-1})^q=a,
\]
contradicting the $q$-indivisibility of $a$ in $G$.

Assume next that $y\in G$. Exactly as in the second half of the proof of Lemma~\ref{lem:one-prime-step}, the fact that $y^q$ fixes the first edge on the geodesic from the base $G$-vertex to $h\cdot v_A$ yields an element $g_0\in G$ such that, after setting $y_1=g_0^{-1}yg_0\in G$,
\[
y_1^q\in \gen{a}.
\]
Say $y_1^q=a^n$. Since also
\[
a^n=g_0^{-1}hah^{-1}g_0,
\]
the elements $a$ and $a^n$ are conjugate in $H$ and both lie in $\gen{a}$. Lemma~\ref{lem:central-conjugacy} therefore gives $g\in G$ with
\[
g^{-1}ag=a^n=y_1^q.
\]
Hence
\[
(gy_1g^{-1})^q=a,
\]
again contradicting the $q$-indivisibility of $a$ in $G$.

Both cases are impossible. Therefore $a$ remains $q$-indivisible in $H$ for every prime $q\neq p$ such that $a$ is $q$-indivisible in $G$.
\end{proof}

\begin{theorem}\label{thm:hom-exact-s42}
Let $G$ be a uniformly prime-indivisible group, and let $\mathcal L$ be any substitution language of group-modal formulas with
\[
\Lgrp\subseteq \mathcal L\subseteq \LgrpM.
\]
Then
\[
\Val_{\GrpHom}(G,\mathcal L_G)=\SFourTwo.
\]
\end{theorem}

\begin{proof}
The lower bound is Theorem~\ref{thm:hom-s42-lower}. The embedding proof uses prime-torsion buttons, but those are not stable under arbitrary homomorphisms. Here I replace them by divisibility buttons on a distinguished parameter. For the upper bound, fix integers $N\geq 2$ and $M\geq 1$. We shall produce an independent family of $M$ buttons together with a dial of length $N$.

For the buttons, choose a finite set $S$ of primes with $|S|=M$. By uniform prime-indivisibility, there is an element $a\in G$ of infinite order such that $a$ is $p$-indivisible in $G$ for every $p\in S$. For each $p\in S$, let
\[
b_p(x):=\exists y\,(y^p=x).
\]
Evaluated at the parameter $a$, the assertion $b_p(a)$ says that the image of $a$ is $p$-divisible. Because $a$ is $p$-indivisible in $G$ for every $p\in S$, these buttons are all initially unpushed at $G$.

We first show that each $b_p(a)$ is a pure button. Let $h\colon G\to H$ be an arbitrary homomorphism out of $G$, write $a_H=h(a)$, and suppose that $a_H$ is not $p$-divisible in $H$.

If $a_H$ has infinite order, apply Lemma~\ref{lem:one-prime-step} with $H$ in place of $G$ and $a_H$ in place of $a$. If instead $a_H$ has finite order $m_H$, then Lemma~\ref{lem:finite-order-divisibility} shows that $p\mid m_H$, since otherwise $a_H$ would already be $p$-divisible in $\gen{a_H}$. In that case apply Lemma~\ref{lem:one-prime-step-finite}. In either situation we obtain a homomorphism from $H$ to a further group in which the image of $a_H$ has acquired a $p$th root. Since $p$-divisibility is preserved by homomorphisms, once pushed the button stays pushed. Thus each $b_p(a)$ is a pure button.

We shall use a simple structural point about these one-prime-step constructions. In both Lemma~\ref{lem:one-prime-step} and Lemma~\ref{lem:one-prime-step-finite}, the canonical map from the source group into the amalgamated free product is injective. Therefore the order of the distinguished element is preserved at each stage. Consequently, once the current image of $a$ has infinite order, every later image still has infinite order; and if at some stage the current image has finite order $m$, then every later image still has order $m$.

Next we verify independence of the family $\{b_p(a):p\in S\}$. Let $h\colon G\to H$ be an arbitrary homomorphism out of $G$, write $a_H=h(a)$, and define
\[
A=\{p\in S:H\models b_p[a_H]\}.
\]
Thus $A$ is the set of already-pushed button primes at $H$, and $S\setminus A$ is the set of currently unpushed ones. Fix any subset $T\subseteq S\setminus A$. Enumerate
\[
T=\{p_1,\dots,p_r\}.
\]
We recursively construct groups $H_0,\dots,H_r$, elements $a_i\in H_i$, and homomorphisms
\[
H=H_0\xrightarrow{\psi_0}H_1\xrightarrow{\psi_1}\cdots\xrightarrow{\psi_{r-1}}H_r
\]
so that for each $i\le r$:
\begin{enumerate}[label=(\roman*)]
    \item $a_i$ is the image of $a_H$ in $H_i$;
    \item $H_i\models b_p[a_i]$ for every $p\in A\cup\{p_1,\dots,p_i\}$; and
    \item $H_i\models \neg b_q[a_i]$ for every $q\in S\setminus\bigl(A\cup\{p_1,\dots,p_i\}\bigr)$.
\end{enumerate}
Set $H_0=H$ and $a_0=a_H$. Suppose $H_i$ and $a_i$ are given, and let $p=p_{i+1}$. By inductive clause~(iii), every prime $q$ not yet chosen is such that $a_i$ is $q$-indivisible in $H_i$, so the relevant one-prime-step lemma is applicable.

If $a_i$ has infinite order, let
\[
H_{i+1}=H_i*_{\gen{a_i}}A_p,
\]
and let $\psi_i\colon H_i\to H_{i+1}$ be the canonical homomorphism. Then Lemma~\ref{lem:one-prime-step} shows that $a_{i+1}=\psi_i(a_i)$ is $p$-divisible in $H_{i+1}$ and remains $q$-indivisible for every prime $q\neq p$ for which $a_i$ was $q$-indivisible in $H_i$.

If instead $a_i$ has finite order $m_i$, then $H_i\models \neg b_p[a_i]$, so Lemma~\ref{lem:finite-order-divisibility} gives $p\mid m_i$. Let
\[
H_{i+1}=H_i*_{\gen{a_i}}C_{pm_i},
\]
where $a_i$ is identified with $t^p$ for a generator $t$ of $C_{pm_i}$, and let $\psi_i\colon H_i\to H_{i+1}$ be the canonical homomorphism. By Lemma~\ref{lem:one-prime-step-finite}, the image $a_{i+1}=\psi_i(a_i)$ is $p$-divisible in $H_{i+1}$ and remains $q$-indivisible for every prime $q\neq p$ for which $a_i$ was $q$-indivisible in $H_i$.

Since divisibility is preserved by homomorphisms, primes already in $A\cup\{p_1,\dots,p_i\}$ stay pushed at stage $i+1$. The preceding paragraph on order-preservation shows that we remain in the same order regime (infinite or finite of fixed order) throughout the recursion. Hence the three inductive requirements continue to hold. At the end we obtain a group $H_r$ satisfying
\[
H_r\models b_p[a_r]\text{ for }p\in A\cup T,
\qquad
H_r\models \neg b_q[a_r]\text{ for }q\in S\setminus(A\cup T).
\]
Thus from the arbitrary homomorphism $h\colon G\to H$ we have pushed exactly the chosen subfamily $T$ of the currently unpushed buttons, while leaving the others unchanged. This proves that the buttons are genuinely independent.

For dials, define $d_1$ to say that the center is trivial. For $1<n<N$, let $d_n$ assert that the center has size exactly $n$, and let $d_{\ge N}$ assert that the center has size at least $N$. Exactly one of
\[
d_1,\dots,d_{N-1},d_{\ge N}
\]
holds in any group. From any group $K$, these dial values are realized by the homomorphisms
\[
K\hookrightarrow K*F_2,
\qquad
K\hookrightarrow (K*F_2)\times C_n\ (1<n<N),
\qquad
K\hookrightarrow (K*F_2)\times C_N.
\]
The first map realizes $d_1$ because the center of a nontrivial free product is trivial; if $K$ is trivial, then $K*F_2\cong F_2$, whose center is also trivial. The other maps realize the remaining dial values because the center of $K*F_2$ is trivial and the center of a direct product is the product of the centers.

It remains to show that these dial moves preserve every button $b_p(a_K)$. Let $k\colon G\to K$ be a homomorphism, and write $a_K=k(a)$. There is a canonical retraction
\[
\rho\colon K*F_2\to K
\]
which is the identity on the factor $K$ and sends the free factor $F_2$ to the identity. Hence, if $u^p=a_K$ in $K*F_2$, then applying $\rho$ gives
\[
\rho(u)^p=\rho(a_K)=a_K
\]
in $K$. Conversely, any $p$th root of $a_K$ already present in $K$ remains a $p$th root in $K*F_2$ under the canonical inclusion. Thus the map
\[
K\hookrightarrow K*F_2
\]
preserves the truth of every button $b_p(a_K)$.

For the direct-product dial moves, suppose
\[
(u,c)^p=(a_K,1)
\]
in $(K*F_2)\times C_n$. Projecting to the first factor yields $u^p=a_K$ in $K*F_2$, and applying the retraction $\rho$ again gives a $p$th root of $a_K$ in $K$. Conversely, if $x^p=a_K$ in $K$, then $(x,1)^p=(a_K,1)$ in $(K*F_2)\times C_n$. Therefore each map
\[
K\hookrightarrow (K*F_2)\times C_n
\]
preserves the truth of every button $b_p(a_K)$.

Finally, let $h\colon G\to H$ be an arbitrary homomorphism, let $T$ be any subset of the currently unpushed buttons at the transported parameter $h(a)$, and let the target dial value be one of $d_1,\dots,d_{N-1},d_{\ge N}$. First push exactly the buttons in $T$ by the sequential construction above, obtaining a homomorphic continuation $G\xrightarrow{h}H\to K$. Then perform the appropriate dial move from $K$. The previous paragraph shows that this second step preserves the entire button pattern. Therefore, from the arbitrary homomorphism $h\colon G\to H$, any chosen subfamily of the currently unpushed buttons can be pushed and, independently, the dial can be set to any desired value without disturbing the button status. This is exactly the independence requirement from Definition~\ref{def:buttons-dials-hom}.

Since $M$ and $N$ were arbitrary, we have produced arbitrarily large finite independent families of genuinely independent buttons together with arbitrarily long finite dials independent of those buttons. Proposition~\ref{prop:general-tools}(5), applied with parameter set $A=G$, therefore yields
\[
\Val_{\GrpHom}(G,\mathcal L_G)\subseteq \SFourTwo.
\]
Together with the lower bound, this proves the theorem.
\end{proof}

\begin{corollary}\label{cor:hom-category-s42}
The exact propositional modal theory of the category of groups under homomorphisms, with formulaic substitutions from $\mathcal L$ allowing parameters, is $\SFourTwo$. More precisely, for every substitution language $\mathcal L$ of group-modal formulas with $\Lgrp\subseteq \mathcal L\subseteq \LgrpM$, the intersection of the theories $\Val_{\GrpHom}(G,\mathcal L_G)$ over all groups $G$ is $\SFourTwo$.
\end{corollary}

\begin{proof}
The lower bound is Theorem~\ref{thm:hom-s42-lower}. For the reverse inclusion, observe that $\Z$ is uniformly prime-indivisible: for every finite set of primes $S$, the element $1\in \Z$ has infinite order and is $p$-indivisible for every $p\in S$. Hence Theorem~\ref{thm:hom-exact-s42} applies to $\Z$, yielding exact validities $\SFourTwo$ there.
\end{proof}

\section{Comparison with the embedding semantics}\label{section:comparison}

A companion paper studies modal group theory in the category of groups and embeddings \cite{WoloszynEmbedding}. The two modal systems share the same ambient language and much of the same expressive-power technology, but the change of accessibility relation has substantial consequences.

First, several expressive-power arguments transfer from embeddings to homomorphisms because the same witnessing target groups can still be used. In particular, the proof of cyclic-subgroup membership under homomorphisms does \emph{not} rely on a general claim that arbitrary homomorphisms preserve the relevant centralizer test. Rather, one applies necessity to the canonical embedding into the relevant HNN extension, and that embedding is itself a homomorphism. This is why the arithmetic and complexity results in Sections~\ref{section:expressive} and~\ref{section:complexity} closely parallel their embedding counterparts.

Second, homomorphisms collapse information in a way that embeddings do not. In particular, every group admits a trivial homomorphism into every group. This immediately forces sentential validities up to $\SFive$, in sharp contrast with the embedding semantics, where sentential validities need not be $\SFive$ and parameter-validities at many groups are exactly $\SFourTwo$ rather than $\SFive$ \cite{WoloszynEmbedding}.

Third, the theorem expressing equality of order under embeddings has no homomorphism analogue of the same form. The conjugacy formula used there becomes trivial under passage to the trivial group. Thus the homomorphism system preserves many of the subgroup-theoretic definability results of the embedding system, but not all of them. The difference is conceptually the same as the one already visible in the category of sets, where allowing arbitrary functions instead of injections drastically changes the modal validities \cite{WSet}.

\printbibliography[heading=mgtbibliography]

\end{document}